\documentclass{amsproc}

\usepackage{hyperref} 
\usepackage{url}      

\newtheorem{theorem}{Theorem}[section]
\newtheorem{lemma}[theorem]{Lemma}
\newtheorem{corollary}[theorem]{Corollary}

\theoremstyle{definition}
\newtheorem{definition}[theorem]{Definition}

\theoremstyle{remark}

\numberwithin{equation}{section}

\newcommand{\C}{\mathbb{C}}
\newcommand{\N}{\mathbb{N}}
\newcommand{\R}{\mathbb{R}}

\newcommand{\cA}{\mathcal{A}}
\newcommand{\cB}{\mathcal{B}}
\newcommand{\cC}{\mathcal{C}}
\newcommand{\cF}{\mathcal{F}}
\newcommand{\cK}{\mathcal{K}}
\newcommand{\cL}{\mathcal{L}}
\newcommand{\cM}{\mathcal{M}}

\newcommand{\eps}{\varepsilon}

\begin{document}
\title[Maximally Modulated Singular Integral Operators]
{Maximally Modulated Singular Integral Operators\\
and their Applications to Pseudodifferential Operators\\
on Banach Function Spaces}

\author{Alexei Yu. Karlovich}
\address{
Centro de Matem\'atica e Aplica\c{c}\~oes (CMA) and
Departamento de Matem\'atica,
Faculdade de Ci\^encias e Tecnologia,
Universidade Nova de Lisboa,
Quinta da Torre,
2829--516 Caparica,
Portugal}
\email{oyk@fct.unl.pt}

\thanks{This work was partially supported by the Funda\c{c}\~ao para a Ci\^encia e a Tecnologia
(Portuguese Foundation for Science and Technology) through the projects
PEst-OE/MAT/UI0297/2014 (Centro de Matem\'atica e Aplica\c{c}\~oes)
and
EXPL/MAT-CAL/0840/2013 (Problemas Variacionais em Espa\c{c}os de Sobolev de Expoente Vari\'avel).}

\subjclass[2010]{Primary 42B20, 47G30; Secondary 46E30, 42B25}
\date{July 1, 2014} 


\keywords{%
Maximally modulated singular integral operator,
Calder\'on-Zygmund operator,
Hilbert transform,
pseudodifferential operator with non-regular symbol,
Banach function space,
variable Lebesgue space.}

\begin{abstract}
We prove that if the Hardy-Littlewood maximal operator is boun\-ded on a
separable Banach function space $X(\mathbb{R}^n)$ and on its associate
space $X'(\mathbb{R}^n)$ and a maximally modulated Calder\'on-Zygmund
singular integral operator $T^\Phi$ is of weak type $(r,r)$ for all
$r\in(1,\infty)$, then $T^\Phi$ extends to a bounded operator on $X(\mathbb{R}^n)$.
This theorem implies the boundedness of the maximally modulated Hilbert
transform on variable Lebesgue spaces $L^{p(\cdot)}(\mathbb{R})$ under natural
assumptions on the variable exponent $p:\mathbb{R}\to(1,\infty)$.
Applications of the above result to the boundedness and compactness
of pseudodifferential operators with $L^\infty(\mathbb{R},V(\mathbb{R}))$-symbols
on variable Lebesgue spaces $L^{p(\cdot)}(\mathbb{R})$ are considered.
Here the Banach algebra $L^\infty(\mathbb{R},V(\mathbb{R}))$ consists of all
bounded measurable $V(\mathbb{R})$-valued functions on $\mathbb{R}$ where
$V(\mathbb{R})$ is the Banach algebra of all functions of bounded total variation.
\end{abstract}

\maketitle
\section{Introduction}
In this paper we will be concerned with the boundedness of maximally modulated
Calder\'on-Zygmund singular integral operators and its applications to the boundedness
of pseudodifferential operators with non-regular symbols on separable Banach function
spaces.

Let us define the main operators we are dealing with.
Let $L_0^\infty(\R^n)$ and $C_0^\infty(\R^n)$ denote the sets of all bounded
functions with compact support and all infinitely differentiable functions
with compact support, respectively.
A Calder\'on-Zygmund operator is a linear operator $T$ which is bounded
on $L^2(\R^n)$ such that for every $f\in L_0^\infty(\R^n)$,
\[
(Tf)(x):=\int_{\R^n}K(x,y)f(y)\,dy
\quad\mbox{for a.e.}\quad
x\in\R^n\setminus\operatorname{supp}f,
\]
where $\operatorname{supp}f$ denotes the support of $f$. The kernel
\[
K:\R^n\times\R^n\setminus\{(x,x):x\in\R^n\}\to\C
\]
is assumed to satisfy the following standard conditions:
\[
|K(x,y)|\le\frac{c_0}{|x-y|^n}
\quad\mbox{for}\quad x\ne y
\]
and
\[
|K(x,y)-K(x,y')|+|K(y,x)-K(y',x)|
\le \frac{c_0|y-y'|^\tau}{|x-y|^{n+\tau}}
\]
for $|x-y|>2|y-y'|$, where $c_0$ and $\tau$ are some positive constants independent of $x,y,y'\in\R^n$
(see, e.g., \cite[Section~8.1.1]{G09}). The most prominent example of Calder\'on-Zygmund
operators is the Hilbert transform defined for $f\in L_0^\infty(\R)$ by
\[
(Hf)(x):=\frac{1}{\pi}\mbox{v.p.}\int_\R\frac{f(y)}{x-y}\,dy
=
\lim_{\eps\to 0}\frac{1}{\pi}\int_{\R\setminus I(x,\eps)}\frac{f(y)}{x-y}\,dy,
\quad x\in\R,
\]
where $I(x,\eps):=(x-\eps,x+\eps)$.

Suppose $\Phi=\{\phi_\alpha\}_{\alpha\in\cA}$ is a family of measurable real-valued
functions indexed by an arbitrary set $\cA$. Then for every $\phi_\alpha\in\cA$,
the modulation operator is defined  by
\[
(\cM^{\phi_\alpha}f)(x):=e^{-i\phi_\alpha(x)}f(x),
\quad
x\in\R^n.
\]
Following \cite{GMS05} (see also \cite{DPL13}), the maximally modulated singular integral
operator $T^\Phi$ of the Calder\'on-Zygmund operator $T$ with respect to the family
$\Phi$ is defined for $f\in L_0^\infty(\R^n)$ by
\[
(T^\Phi f)(x):=\sup_{\alpha\in\cA}|T(\cM^{\phi_\alpha}f)(x)|,
\quad
x\in\R^n.
\]
This definition is motivated by the fact that the maximally modulated Hilbert transform
\[
(\cC f)(x):=(H^\Psi f)(x)
\quad\mbox{with}\quad\Psi:=\{\psi_\alpha(x)=\alpha x:\alpha, x\in\R\}
\]
is closely related to the continuous version of the celebrated Carleson-Hunt theorem on the
a.e. convergence of Fourier series (see, e.g.,
\cite[Chap.~2, Section 2.2]{D91}, \cite[Chap.~11]{G09}, and \cite[Chap.~7]{MS13}).
In \cite{GMS05,DPL13} the operator $\cC$ is called the Carleson operator, however
in \cite[Section 11.1]{G09} and in \cite[Section~7.1]{MS13} this term is used for
two different from $\cC$ and each other operators.

For $f\in C_0^\infty(\R)$, consider the maximal singular integral operator given by
\begin{equation}\label{eq:S*-definition}
(S_*f)(x):=
\sup_{-\infty <a<b<+\infty}\left| (S_{(a,b)}f)(x)\right|,
\quad x\in\R,
\end{equation}
where $S_{(a,b)}f$ is the integral analogue of the partial sum of the Fourier series given by
\[
(S_{(a,b)}f)(x):=
\frac{1}{2\pi}
\int_a^b\widehat{f}(\lambda)e^{ix\lambda}\,d\lambda,
\quad x\in\R,
\]
and
\[
\widehat{f}(\lambda):=(\cF f)(\lambda):=\int_\R f(x)e^{-ix\lambda}dx,
\quad\lambda\in\R,
\]
is the Fourier transform of $f$. It is not difficult to see that if $f\in C_0^\infty(\R)$,
then
\begin{equation}\label{eq:S*-C-pointwise}
(S_*f)(x)\le (\cC f)(x)
\quad\mbox{for a.e.}\quad x\in\R.
\end{equation}

For a suitable function $a$ on $\R\times\R$, a pseudodiffferential operator $a(x,D)$
is defined for a function $f\in C_0^\infty(\R)$ by the iterated integral
\begin{equation}\label{eq:PDO}
(a(x,D)f)(x):=\frac{1}{2\pi}\int_\R\,d\lambda\int_\R a(x,\lambda)e^{i(x-y)\lambda}f(y)\,dy,
\quad
x\in\R.
\end{equation}
The function $a$ is called the symbol of the pseudodifferential operator $a(x,D)$.

Our results on the above mentioned operators will be formulated in terms of the
Hardy-Littlewood maximal function, which we define next. Let $1\le r<\infty$.
Given $f\in L_{\rm loc}^r(\R^n)$, the $r$-th maximal operator is defined by
\[
(M_rf)(x):=\sup_{Q\ni x}\left(\frac{1}{|Q|}\int_Q|f(y)|^r dy\right)^{1/r},\quad x\in\R^n,
\]
where the supremum is taken over all cubes $Q$ containing $x$. Here, and throughout,
all cubes will be assumed to have their sides parallel to the coordinate axes and
$|Q|$ will denote the volume of $Q$. For $r=1$ this is the usual Hardy-Littlewood
maximal operator, which will be denoted by $M$.

Let $f\in L^1_{\rm loc}(\mathbb{R}^n)$.  For a cube
$Q\subset\mathbb{R}^n$, put
\[
f_Q:=\frac{1}{|Q|}\int_Q f(x)dx.
\]
The Fefferman-Stein sharp maximal operator $f\mapsto M^\#f$ is defined by
\[
(M^\#f)(x):=\sup_{Q\ni x}\frac{1}{|Q|}\int_Q|f(x)-f_Q|dx,
\quad x\in\R^n,
\]
where the supremum is taken over all cubes $Q$ containing $x$.

Banach function spaces $X(\R^n)$ will be defined in Section~\ref{sec:preliminaries}.
This is a wide class of spaces including rearrangement-invariant (r.i.) 
Lebesgue, Orlicz, and Lorentz spaces, as well as non-r.i. variable Lebesgue
spaces $L^{p(\cdot)}(\R^n)$. The main feature of these spaces is the so-called
lattice property: if $|f(x)|\le|g(x)|$ for a.e. $x\in\R^n$, then $\|f\|_{X(\R^n)}\le\|g\|_{X(\R^n)}$.
In Section~\ref{sec:preliminaries} we collect preliminaries prepare the proof
of main results given in Section~\ref{sec:main}. Let us briefly describe them.

The boundedness of maximally modulated Calder\'on-Zygmund operators $T^\Phi$
on weighted Lebesgue spaces was studied by Grafakos, Martell, and Soria \cite{GMS05}.
A quantitative version of their results was obtained recently by Di Plinio and Lerner
\cite{DPL13}. We show that a pointwise inequality for the sharp maximal function
of $T^\Phi$ obtained in \cite[Proposition~4.1]{GMS05}, combined with the Fefferman-Stein
inequality for Banach function spaces due to Lerner \cite[Corollary~4.2]{L10},
and with the self-improving property of the Hardy-Littlewood maximal function
on Banach function spaces obtained by Lerner and P\'erez \cite[Corolary~1.3]{LP07},
imply the boundedness of $T^\Phi$ on a separable Banach function space $X(\R^n)$
under the natural assumptions that $M$ is bounded on $X(\R^n)$, $M$ is bounded on
its associate space $X'(\R^n)$, and $T^\Phi$ is of weak type $(r,r)$ for all 
$r\in(1,\infty)$ (see Theorem~\ref{th:main}). Notice that the latter hypothesis 
is satisfied for the maximally modulated Hilbert transform $\cC$. This gives the 
boundedness of $\cC$ on $X(\R)$ (see Corollary~\ref{co:boundedness-C-X}). From here
and the pointwise estimate \eqref{eq:S*-C-pointwise} we also get the boundedness
of the operator $S_*$ on separable Banach functions spaces such that $M$ is bounded
on $X(\R)$ and on $X'(\R)$ (see Lemma~\ref{le:boundedness-S*-X}).

Section~\ref{sec:application} is devoted to applications of the above results
to the boundedness of pseudodifferential operators with non-regular symbols
on Banach function spaces. Note that the boundedness of $a(x,D)$ with smooth
(regular) symbols in H\"ormander's and Miyachi's classes on separable Banach
function spaces was studied in \cite{K-A14} (see also \cite{KS13}). On the other 
hand, Yu.~Karlovich \cite{K-Yu07} introduced the class $L^\infty(\R,V(\R))$ of 
bounded measurable $V(\R)$-valued functions on $\R$ where $V(\R)$ is the Banach 
algebra of all functions of bounded total variation. Symbols in $L^\infty(\R,V(\R))$
may have jump discontinuities in both variables. By using the boundedness of the operator
$S_*$ on $L^p(\R)$ for $1<p<\infty$, Yu.~Karlovich \cite[Theorem~3.1]{K-Yu07} 
proved the boundedness of $a(x,D)$ with $a\in L^\infty(\R,V(\R))$ on $L^p(\R)$
for $1<p<\infty$. Later on he extended this result to weighted Lebesgue spaces $L^p(\R,w)$
with Muckenhoupt weights $w$ (see \cite[Theorem~4.1]{K-Yu12}). One of the important
ingredients of those proofs is the pointwise inequality
\[
|(a(x,D)f)(x)|\le 2(S_*f)(x)\|a(x,\cdot)\|_V,
\quad x\in\R,
\]
for $f\in C_0^\infty(\R)$. From this inequality and the boundedness of $S_*$
we obtain the boundedness of $a(x,D)$ on separable Banach function spaces
$X(\R)$ under the assumption that $M$ is bounded on $X(\R)$ and $X'(\R)$
(see Theorem~\ref{th:boundedness-PDO-X}).

In Section~\ref{sec:VLS-boundedness} we specify the above results to the case 
of variable Lebesgue spaces $L^{p(\cdot)}(\R^n)$. If the variable exponent
$p:\R^n\to[1,\infty]$ is bounded away from one and infinity, then in view of 
Diening's theorem \cite[Theorem~8.1]{D05}, the boundedness of $M$ is equivalent
to the boundedness of $M$ on its associate space. Hence all above results have 
simpler formulations in the case of variable Lebesgue spaces (see 
Section~\ref{sec:VLS-boundedness}).

We conclude the paper with a sufficient condition for the compactness of
pseudodifferential operators $a(x,D)$ with $L^\infty(\R,V(\R))$-symbols on
variable Lebesgue spaces $L^{p(\cdot)}(\R)$ (see Corollary~\ref{co:compactness-PDO}).
This result is obtained from Yu.~Karlovich's result \cite[Theorem~4.1]{K-Yu07}
for standard Lebesgue spaces by transferring the compactness property from standard
to variable Lebesgue spaces with the aid of the Krasnosel'skii-type interpolation
theorem (see Section~\ref{sec:transfer-compactness}).
\section{Preliminaries}\label{sec:preliminaries}
\subsection{Banach function spaces}
The set of all Lebesgue measurable complex-valued functions on $\R^n$ is
denoted by $\cM$. Let $\cM^+$ be the subset of functions in $\cM$ whose
values lie  in $[0,\infty]$. The characteristic function of a measurable
set $E\subset\R^n$ is denoted by $\chi_E$ and the Lebesgue measure of $E$
is denoted by $|E|$.
\begin{definition}[{\cite[Chap.~1, Definition~1.1]{BS88}}]
\label{def-BFS}
A mapping $\rho:\cM^+\to [0,\infty]$ is
called a {\it Banach function norm} if, for all functions $f,g,
f_n \ (n\in\N)$ in $\cM^+$, for all constants $a\ge 0$, and for
all measurable subsets $E$ of $\R^n$, the following properties
hold:
\begin{eqnarray*}
{\rm (A1)} & & \rho(f)=0  \Leftrightarrow  f=0\ \mbox{a.e.}, \quad
\rho(af)=a\rho(f), \quad
\rho(f+g) \le \rho(f)+\rho(g),\\
{\rm (A2)} & &0\le g \le f \ \mbox{a.e.} \ \Rightarrow \ \rho(g)
\le \rho(f)
\quad\mbox{(the lattice property)},\\
{\rm (A3)} & &0\le f_n \uparrow f \ \mbox{a.e.} \ \Rightarrow \
       \rho(f_n) \uparrow \rho(f)\quad\mbox{(the Fatou property)},\\
{\rm (A4)} & & |E|<\infty \Rightarrow \rho(\chi_E) <\infty,\\
{\rm (A5)} & & |E|<\infty \Rightarrow \int_E f(x)\,dx \le C_E\rho(f)
\end{eqnarray*}
with $C_E \in (0,\infty)$ which may depend on $E$ and $\rho$ but is
independent of $f$.
\end{definition}
When functions differing only on a set of measure zero are identified,
the set $X(\R^n)$ of all functions $f\in\cM$ for which $\rho(|f|)<\infty$ is
called a \textit{Banach function space}. For each $f\in X(\R^n)$, the norm of
$f$ is defined by
\[
\|f\|_{X(\R^n)} :=\rho(|f|).
\]
The set $X(\R^n)$ under the natural linear space operations and under this norm
becomes a Banach space (see \cite[Chap.~1, Theorems~1.4 and~1.6]{BS88}).

The norm of a bounded sublinear operator $A$ on a Banach function space
$X(\R^n)$ will be denoted by $\|A\|_{\cB(X(\R^n))}$.

If $\rho$ is a Banach function norm, its associate norm $\rho'$ is
defined on $\cM^+$ by
\[
\rho'(g):=\sup\left\{
\int_{\R^n} f(x)g(x)\,dx \ : \ f\in \cM^+, \ \rho(f) \le 1
\right\}, \quad g\in \cM^+.
\]
It is a Banach function norm itself \cite[Chap.~1, Theorem~2.2]{BS88}.
The Banach function space $X'(\R^n)$ determined by the Banach function norm
$\rho'$ is called the \textit{associate space} (\textit{K\"othe dual}) of $X(\R^n)$.
The Lebesgue space $L^p(\R^n)$, $1\le p\le\infty$, are the archetypical
example of Banach function spaces. Other classical examples of Banach function
spaces are Orlicz spaces, rearrangement-invariant spaces, and variable Lebesgue
spaces $L^{p(\cdot)}(\R^n)$.
\subsection{Density of bounded and smooth compactly supported functions in separable Banach function spaces}
The proof of the following fact is standard. For details, see \cite[Lemma~2.10(b)]{KS14},
where it was proved for $n=1$. The proof for arbitrary $n$ is a minor modification
of that one.
\begin{lemma}\label{le:density}
The sets $L_0^\infty(\R^n)$ and $C_0^\infty(\R^n)$ are dense in a separable Banach function space $X(\R^n)$.
\end{lemma}
\subsection{Nonnegative sublinear operators on Banach function spaces}
The operators $T^\Phi$, $\cC$, and $S_{(a,b)}$ although nonlinear, are examples of sublinear
operators that assume only nonnegative values. Let us give a precise definition of this class
of operators.
Let $D$ be a linear subspace of $\cM$. An operator $T:D\to\cM$ is said to be nonnegative
sublinear (cf. \cite[p.~230]{BS88}) if
\begin{equation}\label{eq:nonnegative-sublinear}
0\le T(f+g)\le Tf+Tg,
\quad
T(\lambda f)=|\lambda|Tf
\quad\mbox{a.e. on }\R^n
\end{equation}
for all $f,g\in D$ and all constants $\lambda\in\C$. The following result is well known
for linear operators. With the property
\[
|Tf-Tg|\le |T(f-g)|=T(f-g),
\quad f,g\in D,
\]
which is an immediate consequence of \eqref{eq:nonnegative-sublinear}, essentially the
same proof establishes the result also for nonnegative sublinear operators.
\begin{lemma}\label{le:extension}
Let $D$ be a dense linear subspace of a Banach function space $X(\R^n)$ and $T:D\to\cM$
be a nonnegative sublinear operator. If there exists a positive constant $C$ such that
\[
\|Tf\|_{X(\R^n)}\le C\|f\|_{X(\R^n)}
\quad\mbox{for all}\quad f\in D,
\]
then $T$ has a unique extension to a nonnegative sublinear operator $\widetilde{T}:X(\R^n)\to\cM$
such that
\[
\|\widetilde{T}f\|_{X(\R^n)}\le C\|f\|_{X(\R^n)}
\quad\mbox{for all}\quad f\in X(\R^n).
\]
\end{lemma}
In what follows we will use the same notation for an operator defined
on a dense subspace and for its bounded extension to the whole space.
\subsection{Self-improving property of maximal operators on Banach function spaces}
\label{subsec:self-improving}
If $1<q<\infty$, then from the H\"older inequality one can immediately get that
\[
(Mf)(x)\le (M_rf)(x)\quad \quad\mbox{for a.e.}\quad x\in\R^n.
\]
Thus, the boundedness of any $M_r$, $1<r<\infty$, on a Banach function space
$X(\R^n)$  immediately implies the boundedness of $M$. A partial converse
of this fact, called a \textit{self-improving property} of the Hardy-Littlewood
maximal operator, is also true. It was proved by Lerner and P\'erez \cite{LP07}
(see also \cite{LO10} for another proof)
in a more general setting of quasi-Banach function spaces.
\begin{theorem}[{\cite[Corollary~1.3]{LP07}}]
\label{th:Lerner-Perez}
Let $X(\R^n)$ be a Banach function space. Then $M$ is bounded on $X(\R^n)$ if
and only if $M_r$ is bounded on $X(\R^n)$ for some $r\in(1,\infty)$.
\end{theorem}
\subsection{The Fefferman-Stein inequality for Banach function spaces}
\label{subsec:Fefferman-Stein}
 It is obvious that $M^\#f$ is pointwise dominated by $Mf$. Hence,
by Axiom (A2),
\[
\|M^\#f\|_{X(\R^n)}\le {\rm const}\|f\|_{X(\R^n)}
\quad\mbox{for}\quad f\in X(\R^n)
\]
whenever $M$ is bounded on $X(\R^n)$. The converse inequality for Lebesgue spaces
$L^p(\R^n)$, $1<p<\infty$, was proved by Fefferman and Stein (see
\cite[Theorem~5]{FS72} and also
\cite[Chap.~IV, Section~2.2]{S93}). The following extension of the Fefferman-Stein
inequality to Banach function spaces was proved by Lerner \cite{L10}.

Let $S_0(\R^n)$ be the space of all measurable functions $f$ on $\R^n$ such that
\[
|\{x\in\R^n:|f(x)|>\lambda\}|<\infty
\quad\mbox{for all}\quad \lambda>0.
\]
\begin{theorem}[{\cite[Corollary~4.2]{L10}}]
\label{th:Lerner}
Let $M$ be bounded on a Banach function space $X(\R^n)$. Then $M$ is bounded
on its associate space $X'(\R^n)$ if and only
if there exists a constant $C_\#>0$ such that, for all $f\in S_0(\R^n)$,
\[
\|f\|_{X(\R^n)}\le C_\#\|M^\#f\|_{X(\R^n)}.
\]
\end{theorem}
\subsection{Pointwise inequality for the sharp maximal function of $T^\Phi$}
Let $1\le r<\infty$. Recall that a sublinear operator $A:L^r(\R^n)\to\cM$ is said to be of weak type $(r,r)$
if
\[
|\{x\in\R^n:|(Af)(x)|>\lambda\}|\le \frac{C^r}{\lambda^r}\int_{\R^n}|f(y)|^r\,dy
\]
for all $f\in L^r(\R^n)$ and $\lambda>0$, where $C$ is a positive constant independent
of $f$ and $\lambda$. It is well known that if $A$ is bounded on the standard Lebesgue
space $L^r(\R^n)$, then it is of weak type $(r,r)$.

Grafakos, Martell, and Soria \cite{GMS05} developed two alternative approaches to weighted
$L^r$ estimates for maximally modulated Calder\'on-Zygmund singular integral operators $T^\Phi$.
One is based on good-$\lambda$ inequalities, another rests on the following pointwise
estimate for the sharp maximal function of $T^\Phi$.
\begin{lemma}[{\cite[Proposition~4.1]{GMS05}}]
\label{le:GMS}
Suppose $T$ is a Calder\'on-Zygmund operator and $\Phi=\{\phi_\alpha\}_{\alpha\in\cA}$
is a family of measurable real-valued functions indexed by an arbitrary set $\cA$.
If $T^\Phi$ is of weak type $(r,r)$ for some $r\in(1,\infty)$,
then there is a positive constant
$C_r$ such that for every $f\in L_0^\infty(\R^n)$,
\[
M^\#(T^\Phi f)(x)\le C_r M_rf(x)\quad\mbox{for a.e.}\quad x\in\R^n.
\]
\end{lemma}
\section{Maximally modulated singular integrals on Banach function spaces}\label{sec:main}
\subsection{Boundedness of maximally modulated Calder\'on-Zygmund singular integral operators on Banach function spaces}
We are in a position to prove the main result of the paper.
\begin{theorem}\label{th:main}
Let $X(\R^n)$ be a separable Banach function space. Suppose the Hardy-Littlewood maximal
operator $M$ is bounded on $X(\R^n)$ and on its associate space $X'(\R^n)$.
Suppose $T$ is a Calder\'on-Zygmund operator and $\Phi=\{\phi_\alpha\}_{\alpha\in\cA}$
is a family of measurable real-valued functions indexed by an arbitrary set $\cA$.
If $T^\Phi$ is of weak type $(r,r)$ for all $r\in(1,\infty)$, then $T^\Phi$ extends to
a bounded operator on the space $X(\R^n)$.
\end{theorem}
\begin{proof}
We argue as in the proof of \cite[Theorem~1.2]{K-A14}.
Since $M$ is bounded on $X(\R^n)$, by Theorem~\ref{th:Lerner-Perez}, there is an $r\in(1,\infty)$
such that the maximal function $M_r$ is bounded on $X(\R^n)$, that is, there is a positive constant
$C$ such that
\begin{equation}\label{eq:main-1}
\|(M_r\varphi)\|_{X(\R^n)}\le C\|\varphi\|_{X(\R^n)}
\quad\mbox{for all}\quad \varphi\in X(\R^n).
\end{equation}
Assume that $f\in C_0^\infty(\R^n)$. By the hypothesis, $T^\Phi$ is of weak type $(r,r)$ and $M$
is bounded  on $X'(\R^n)$. Therefore, $T^\Phi f\in S_0(\R^n)$. Moreover, by Theorem~\ref{th:Lerner},
there exists a positive constant $C_\#$ such that
\begin{equation}\label{eq:main-2}
\|T^\Phi f\|_{X(\R^n)}\le C_\# \|M^\#(T^\Phi f)\|_{X(\R^n)}
\quad\mbox{for all}\quad
f\in C_0^\infty(\R^n).
\end{equation}
From Lemma~\ref{le:GMS} and Axioms (A1)--(A2) we conclude that there exists a positive constant $C_r$
such that
\begin{equation}\label{eq:main-3}
\|M^\#(T^\Phi f)\|_{X(\R^n)}
\le C_r\|M_r f\|_{X(\R^n)}
\quad\mbox{for all}\quad
f\in C_0^\infty(\R^n).
\end{equation}
Combining inequalities \eqref{eq:main-1}--\eqref{eq:main-3}, we arrive at
\[
\|T^\Phi f\|_{X(\R^n)}\le C C_\# C_r\|f\|_{X(\R^n)}
\quad\mbox{for all}\quad
f\in C_0^\infty(\R^n).
\]
To conclude the proof, it remains to recall that $C_0^\infty(\R^n)$ is dense in the separable
Banach function space $X(\R^n)$ in view of Lemma~\ref{le:density} and apply Lemma~\ref{le:extension}.
\end{proof}
\subsection{Boundedness of the maximally modulated Hilbert transform on standard Lebesgue spaces}
Fix $f\in L_{loc}^1(\R)$. Let $H_*$ be the maximal Hilbert transform given by
\[
(H_*f)(x):=\sup_{\eps>0}\left|\frac{1}{\pi}\int_{\R\setminus I(x,\eps)}\frac{f(y)}{x-y}\,dy\right|,
\]
where $I(x,\eps)=(x-\eps,x+\eps)$. Further, let $\cC_*$ be the maximally modulated maximal Hilbert
transform (called also the maximal Carleson operator) defined by
\[
(\cC_*f)(x):=\sup_{a\in\R}(H_*(\cM^{\psi_a} f))(x)
\quad\mbox{with}\quad
\psi_a(x)=a x,\quad a, x\in\R.
\]
It is easy to see that
\begin{equation}\label{eq:pointwise}
(\cC f)(x)\le(\cC_* f)(x),
\quad x\in\R.
\end{equation}
The boundedness of the operator $\cC_*$ on the standard Lebesgue spaces $L^r(\R)$ is proved, e.g.,
in \cite[Theorem~11.3.3]{G09} (see also \cite[Theorem~2.7]{K-Yu12}). From this observation and
\eqref{eq:pointwise} we get the following result (see also \cite[Theorem~2.1]{D91} and
\cite[Theorem~2.8]{K-Yu12}).
\begin{lemma}\label{le:boundedness-C-Lr}
The maximally modulated Hilbert transform $\cC$ is bounded on every standard Lebesgue space $L^r(\R)$ for $1<r<\infty$.
\end{lemma}
\subsection{Boundedness of the maximally modulated Hilbert transform on separable Banach function spaces}
From Theorem~\ref{th:main} and Lemma~\ref{le:boundedness-C-Lr} we immediately get the following.
\begin{corollary}\label{co:boundedness-C-X}
Let $X(\R)$ be a separable Banach function space. Suppose the Hardy-Littlewood maximal
operator $M$ is bounded on $X(\R)$ and on its associate space $X'(\R)$. Then the maximally
modulated Hilbert transform $\cC$ extends to a bounded operator on $X(\R)$.
\end{corollary}
\section{Boundedness of pseudodifferential operators\\ with non-regular symbols on Banach function spaces}
\label{sec:application}
\subsection{Functions of bounded total variation}
Let $a$ be a complex-valued function of bounded total variation $V(a)$ on $\R$ where
\[
V(a):=\sup\left\{
\sum_{k=1}^n |a(x_k)-a(x_{k-1})|: -\infty<x_0<x_1<\dots<x_n<+\infty, n\in\N
\right\}
\]
Hence at every point $x\in\dot{\R}:=\R\cup\{\infty\}$ the one-sided limits
\[
a(x\pm 0)=\lim_{t\to x^\pm}a(t)
\]
exist, where $a(\pm\infty)=a(\infty\mp 0)$, and the set of discontinuities of $a$ is at most
countable (see, e.g., \cite[Chap. VIII, Sections 3 and 9]{N55}).
Without loss of generality we will assume that functions of bounded total variation
are continuous from the left at every discontinuity point $x\in\dot{R}$.
The set $V(\R)$ of all continuous from the left functions of bounded total variation
on $\R$ is a unital non-separable Banach algebra with the norm
\[
\|a\|_V:=\|a\|_{L^\infty(\R)}+V(a).
\]

By analogy with $V(a)=V_{-\infty}^{+\infty}(a)$, one can define the total variations
$V_c^d(a)$, $V_{-\infty}^c(a)$, and $V_d^{+\infty}(a)$ of a function $a:\R\to\C$ on
$[c,d]$, $(-\infty,c]$, and $[d,+\infty)$, taking, respectively, the partitions
\[
c=x_0<x_1<\dots<x_n=d,
\quad
-\infty<x_0<x_1<\dots<x_n=c,
\]
and $d=x_0<x_1<\dots<x_n<+\infty$.
\subsection{Non-regular symbols of pseudodifferential operators}
Following \cite{K-Yu07,K-Yu12}, we denote by $L^\infty(\R,V(\R))$ the set of
functions $a:\R\times\R\to\C$ such that $\widehat{a}:x\mapsto a(x,\cdot)$ is a bounded
measurable $V(\R)$-valued function on $\R$. Note that in view of non-separability of
the Banach space $V(\R)$, the measurability of $\widehat{a}$ means that the map
$\widehat{a}:\R\to V(\R)$ possesses the Luzin property: for any compact set $K\subset\R$
and any $\delta$ there is a compact set $K_\delta\subset K$ such that $|K\setminus K_\delta|<\delta$
and $\widehat{a}$ is continuous on $K_\delta$ (see, e.g., \cite[Chap. IV, Section 4, p.~487]{S67}).
This implies that the function $x\mapsto a(x,\lambda\pm 0)$ for all $\lambda\in\dot{\R}$
and the function $x\mapsto\|a(x,\cdot)\|_V$ are measurable on $\R$ as well. Note that
for almost all $x\in\R$ the limits $a(x,\lambda\pm 0)$ the limits exist for all $\lambda\in\dot{\R}$,
$a(x,\lambda)=a(x,\lambda-0)$ for all $\lambda\in\R$ and we put
\[
a(x,\pm\infty):=\lim_{\lambda\to\pm\infty}a(x,\lambda).
\]
Therefore, the functions $a(\cdot,\lambda\pm 0)$ for every $\lambda\in\dot{\R}$ and the function
$x\mapsto\|a(x,\cdot)\|_V$, where
\[
\|a(x,\cdot)\|_V:=\|a(x,\cdot)\|_{L^\infty(\R)}+V(a(x,\cdot)),
\]
belong to $L^\infty(\R)$. Clearly, $L^\infty(\R, V(\R))$ is a unital Banach algebra with
the norm
\[
\|a\|_{L^\infty(\R,V(\R))}=\operatornamewithlimits{ess\,sup}_{x\in\R}\|a(x,\cdot)\|_V.
\]
\subsection{Pointwise inequality for pseudodifferential operators}
The following pointwise estimate was obtained by Yuri Karlovich in the proof of
\cite[Theorem~3.1]{K-Yu07} and \cite[Theorem~4.1]{K-Yu12}.
\begin{lemma}\label{le:PDO-pointwise}
If $a\in L^\infty(\R,V(\R))$ and $f\in C_0^\infty(\R)$, then
\[
\big|(a(x,D)f)(x)\big|\le 2(S_*f)(x)\|a(x,\cdot)\|_V
\quad\mbox{for a.e.}\quad x\in\R.
\]
\end{lemma}
\subsection{Boundedness of the maximal singular integral operator $S_*$ on separable Banach function spaces}
We continue with the following result on the boundedness of the maximal singular
integral operator $S_*$ initially defined for $f\in C_0^\infty(\R)$ by \eqref{eq:S*-definition}.
\begin{lemma}\label{le:boundedness-S*-X}
Let $X(\R)$ be a separable Banach function space. Suppose the Hardy-Littlewood maximal
operator $M$ is bounded on $X(\R)$ and on its associate space $X'(\R)$. Then the
operator $S_*$, defined for the functions $f\in C_0^\infty(\R)$ by \eqref{eq:S*-definition},
extends to a bounded operator on the space $X(\R)$.
\end{lemma}
\begin{proof}
Fix $f\in C_0^\infty(\R)$. It is not difficult to check (see, e.g., \cite[Chap.~2, Section 2.2]{D91}
and also \cite[p.~475]{G09}) that
\[
(S_{(a,b)}f)(x)=\frac{i}{2}\left\{\cM^{-\psi_a}\big(H(\cM^{\psi_a}f)\big)(x)-\cM^{-\psi_b}\big(H(\cM^{\psi_b}f)\big)(x)\right\},
\quad x\in\R,
\]
where $\psi_a(x)=ax$, $\psi_b(x)=bx$ and $-\infty<a<b<+\infty$. Therefore,
\begin{align*}
(S_* f)(x)
&=\sup_{-\infty <a<b<\infty}|(S_{(a,b)}f)(x)|
\\
&\le
\frac{1}{2}\sup_{a\in\R}\left|\big(H(\cM^{\psi_a}f)\big)(x)\right|
+
\frac{1}{2}\sup_{b\in\R}\left|\big(H(\cM^{\psi_b}f)\big)(x)\right|
\\
&=(\cC f)(x),
\quad x\in\R.
\end{align*}
From this inequality, Axioms (A1)-(A2), and Corollary~\ref{co:boundedness-C-X} we get
\[
\|S_*f\|_{X(\R)}\le\|\cC f\|_{X(\R)}\le\|\cC\|_{\cB(X(\R))}\|f\|_{X(\R)}
\quad\mbox{for}\quad f\in C_0^\infty(\R).
\]
It remains to apply Lemma~\ref{le:density}.
\end{proof}
\subsection{Boundedness of pseudodifferential operators with $L^\infty(\R,V(\R))$ symbols on Banach function spaces}
We are ready to prove the boundedness result for pseudodifferential operators with
non-regular symbols on separable Banach function spaces.
\begin{theorem}\label{th:boundedness-PDO-X}
Let $X(\R)$ be a separable Banach function space. Suppose the Hardy-Littlewood maximal
operator $M$ is bounded on $X(\R)$ and on its associate space $X'(\R)$. If
$a\in L^\infty(\R,V(\R))$, then the pseudodifferential operator $a(x,D)$, defined for
the functions $f\in C_0^\infty(\R)$ by the iterated integral \eqref{eq:PDO},
extends to a bounded linear operator on the space $X(\R)$ and
\[
\|a(x,D)\|_{\cB(X(\R))} \le 2\|S_*\|_{\cB(X(\R))}\|a\|_{L^\infty(\R,V(\R))}.
\]
\end{theorem}
\begin{proof}
From Lemma~\ref{le:PDO-pointwise}, axioms (A1)-(A2), and Lemma~\ref{le:boundedness-S*-X} we
obtain for $f\in C_0^\infty(\R)$,
\begin{align*}
\|a(x,D)f\|_{X(\R)}
&\le
2\|S_*f\|_{X(\R)}\operatornamewithlimits{ess\,sup}_{x\in\R}\|a(x,\cdot)\|_V
\\
&\le
2\|S_*\|_{\cB(X(\R))}\|a\|_{L^\infty(\R,V(\R))}\|f\|_{X(\R)}.
\end{align*}
Since $C_0^\infty(\R)$ is dense in the space $X(\R)$ in view of Lemma~\ref{le:density},
from the above estimate we arrive immediately at the desired conclusion.
\end{proof}
\section{Boundedness of maximally modulated Calder\'on-Zygmund operators \\
and pseudodifferential operators with non-regular symbols\\ on variable Lebesgue spaces}
\label{sec:VLS-boundedness}
\subsection{Variable Lebesgue spaces}\label{subsec:VLE}
Let $p\colon\R^n\to[1,\infty]$ be a measurable a.e. finite function. By
$L^{p(\cdot)}(\R^n)$ we denote the set of all complex-valued functions
$f$ on $\R$ such that
\[
I_{p(\cdot)}(f/\lambda):=\int_{\R^n} |f(x)/\lambda|^{p(x)} dx <\infty
\]
for some $\lambda>0$. This set becomes a Banach function space when
equipped with the norm
\[
\|f\|_{p(\cdot)}:=\inf\big\{\lambda>0: I_{p(\cdot)}(f/\lambda)\le 1\big\}.
\]
It is easy to see that if $p$ is constant, then $L^{p(\cdot)}(\R^n)$ is nothing but
the standard Lebesgue space $L^p(\R^n)$. The space $L^{p(\cdot)}(\R^n)$
is referred to as a variable Lebesgue space.

We will always suppose that
\begin{equation}\label{eq:exponents}
1<p_-:=\operatornamewithlimits{ess\,inf}_{x\in\R^n}p(x),
\quad
\operatornamewithlimits{ess\,sup}_{x\in\R^n}p(x)=:p_+<\infty.
\end{equation}
Under these conditions, the space $L^{p(\cdot)}(\R^n)$ is
separable and reflexive, and its associate space is isomorphic to
$L^{p'(\cdot)}(\R^n)$, where
\[
1/p(x)+1/p'(x)=1
\quad\mbox{for a.e.}\quad
x\in\R^n
\]
(see e.g. \cite[Chap.~2]{CF13} or \cite[Chap.~3]{DHHR11}).
\subsection{The Hardy-Littlewood maximal function on variable Lebesgue spaces}
\label{subsec:M-VLE}
By $\cB_M(\R^n)$ denote the set of all measurable functions $p:\R^n\to[1,\infty]$
such that \eqref{eq:exponents} holds and the Hardy-Littlewood
maximal operator is bounded on the variable Lebesgue space $L^{p(\cdot)}(\R^n)$.

To provide a simple sufficient conditions guaranteeing that $p\in\cB_M(\R^n)$,
we need the following definition. Given a function $r:\R^n\to\R$, one says that
$r$ is locally log-H\"older continuous if there exists a constant $C_0>0$ such that
\[
|r(x)-r(y)|\le \frac{C_0}{-\log|x-y|}
\]
for all $x,y\in\R^n$ such that $|x-y|<1/2$. One says that $r:\R^n\to\R$ is log-H\"older
continuous at infinity if there exist constants $C_\infty$ and $r_\infty$ such that
for all $x\in\R^n$,
\[
|r(x)-r_\infty|\le\frac{C_\infty}{\log(e+|x|)}.
\]
The class of functions $r:\R^n\to\R$ that are simultaneously locally log-H\"older
continuous and log-H\"older continuous at infinity is denoted by $LH(\R^n)$.
From \cite[Proposition~2.3 and Theorem~3.16]{CF13} we extract the following.
\begin{theorem}
Let $p\in LH(\R^n)$ satisfy \eqref{eq:exponents}. Then $p\in\cB_M(\R^n)$.
\end{theorem}
Although the latter result provides a nice sufficient condition for the boundedness
of the Hardy-Littlewood maximal operator on the variable Lebesgue space $L^{p(\cdot)}(\R^n)$,
it is not necessary. Notice that all functions in $LH(\R^n)$ are continuous and have
limits at infinity. Lerner \cite{L05} (see also \cite[Example~4.68]{CF13})
proved that if $p_0>1$ and $\mu\in\R$ is sufficiently close to zero, then
the following variable exponent
\[
p(x)=p_0+\mu\sin(\log\log(1+\max\{|x|,1/|x|\})),
\quad x\ne 0,
\]
belongs to $\cB_M(\R)$. It is clear that the function $p$ does not have limits
at zero or infinity. We refer to the recent monographs \cite{CF13,DHHR11} for further
discussions concerning the fascinating and still mysterious class $\cB_M(\R^n)$.

We will need the following remarkable result proved by Diening \cite[Theorem~8.1]{D05}
(see also \cite[Theorem~5.7.2]{DHHR11} and \cite[Corollary~4.64]{CF13}).
\begin{theorem}\label{th:Diening}
We have $p\in\cB_M(\R^n)$ if and only if $p'\in\cB_M(\R^n)$.
\end{theorem}
\subsection{Boundedness of maximally modulated Calder\'on-Zygmund singular integral 
operators on variable Lebesgue spaces}
From Theorems~\ref{th:main} and~\ref{th:Diening} we immediately get the following.
\begin{corollary}\label{co:boundedness-T-phi-Lp}
Let $p\in\cB_M(\R^n)$.
Suppose $T$ is a Calder\'on-Zygmund operator and $\Phi=\{\phi_\alpha\}_{\alpha\in\cA}$
is a family of measurable real-valued functions indexed by an arbitrary set $\cA$.
If $T^\Phi$ is of weak type $(r,r)$ for all $r\in(1,\infty)$, then $T^\Phi$ extends to
a bounded operator on the variable Lebesgue space $L^{p(\cdot)}(\R^n)$.
\end{corollary}
In turn, Corollary~\ref{co:boundedness-T-phi-Lp} and Lemma~\ref{le:boundedness-C-Lr} yield
the following.
\begin{corollary}\label{co:boundedness-C-Lp}
If $p\in\cB_M(\R)$, then the maximally modulated Hilbert transform $\cC$ extends to a
bounded operator on the variable Lebesgue space $L^{p(\cdot)}(\R)$.
\end{corollary}
\subsection{Boundedness of pseudodifferential operators with $L^\infty(\R,V(\R))$ symbols on variable Lebesgue spaces}
Combining Lemma~\ref{le:boundedness-S*-X} with Theorem~\ref{th:Diening} we arrive at the following.
\begin{corollary}\label{co:boundedness-S*-Lp}
Suppose $p\in\cB_M(\R)$. Then the operator $S_*$, defined for the functions
$f\in C_0^\infty(\R)$ by \eqref{eq:S*-definition}, extends to a bounded operator on the
variable Lebesgue space $L^{p(\cdot)}(\R)$.
\end{corollary}
From Theorems~\ref{th:boundedness-PDO-X} and \ref{th:Diening}, taking into account
Corollary~\ref{co:boundedness-S*-Lp}, we get the following.
\begin{corollary}\label{co:boundedness-PDO-Lp}
If $p\in\cB_M(\R)$ and $a\in L^\infty(\R,V(\R))$, then the pseudodifferential
operator $a(x,D)$, defined for the functions $f\in C_0^\infty(\R)$ by the iterated integral
\eqref{eq:PDO}, extends to a bounded linear operator on the space $L^{p(\cdot)}(\R)$ and
\[
\|a(x,D)\|_{\cB(L^{p(\cdot)}(\R))} \le 2\|S_*\|_{\cB(L^{p(\cdot)}(\R))}\|a\|_{L^\infty(\R,V(\R))}.
\]
\end{corollary}
\section{Compactness of pseudodifferential operators with non-regular symbols on variable Lebesgue spaces}
\subsection{Compactness of pseudodifferential operators with $L^\infty(\R,V(\R))$ symbols on standard Lebesgue spaces}
We start with the case of constant exponents.
\begin{theorem}[{\cite[Theorem~4.1]{K-Yu07}}]
\label{th:compactness-PDO-Lr}
Let $1<r<\infty$. If $a\in L^\infty(\R,V(\R))$ and
\begin{enumerate}
\item[{\rm(a)}] $a(x,\pm\infty)=0$ for almost all $x\in\R;$
\item[{\rm(b)}] $\lim\limits_{|x|\to\infty}V(a(x,\cdot))=0;$
\item[{\rm(c)}] for every $N>0$,
\[
\lim_{L\to+\infty}\operatornamewithlimits{ess\,sup}_{|x|\le N} \left(V_{-\infty}^{-L}(a(x,\cdot))+V_L^{+\infty}(a(x,\cdot))\right)=0;
\]
\end{enumerate}
then the pseudodifferential operator $a(x,D)$ is compact on the standard Lebesgue space $L^r(\R)$.
\end{theorem}
\subsection{Transferring the compactness property from standard Lebesgue spaces to variable Lebesgue spaces}
\label{sec:transfer-compactness}
For a Banach space $E$, let $\cL(E)$ and $\cK(E)$ denote the Banach
algebra of all bounded linear operators and its ideal of all compact
operators on $E$, respectively.
\begin{theorem}\label{th:interpolation}
Let $p_j:\R^n\to[1,\infty]$, $j=0,1$, be a.e. finite measurable
functions, and let $p_\theta:\R^n\to[1,\infty]$ be defined for
$\theta\in[0,1]$ by
\[
\frac{1}{p_\theta(x)}=\frac{\theta}{p_0(x)}+\frac{1-\theta}{p_1(x)},\quad
x\in\R^n.
\]
Suppose $A$ is a linear operator defined on $L^{p_0(\cdot)}(\R^n)+L^{p_1(\cdot)}(\R^n)$.
\begin{enumerate}
\item[{\rm(a)}] If $A\in\cL(L^{p_j(\cdot)}(\R^n))$ for $j=0,1$, then
    $A\in\cL(L^{p_\theta(\cdot)}(\R^n))$ for all $\theta\in[0,1]$ and
\[
\|A\|_{\cL(L^{p_\theta(\cdot)}(\R^n))} \le 4
\|A\|_{\cL(L^{p_0(\cdot)}(\R^n))}^\theta
\|A\|_{\cL(L^{p_1(\cdot)}(\R^n))}^{1-\theta}.
\]

\item[{\rm(b)}]   If
    $A\in\cK(L^{p_0(\cdot)}(\R^n))$ and
    $A\in\cL(L^{p_1(\cdot)}(\R^n))$, then
    $A\in\cK(L^{p_\theta(\cdot)}(\R^n))$ for all $\theta\in(0,1)$.
\end{enumerate}
\end{theorem}
Part (a) is proved in \cite[Corollary~7.1.4]{DHHR11} under the
more general assumption that $p_j$ may take infinite values on
sets of positive measure (and in the setting of arbitrary measure
spaces). Part (b) follows from a general interpolation theorem by
Cobos, K\"uhn, and Schonbeck \cite[Theorem~3.2]{CKS92} for the
complex interpolation method for Banach lattices satisfying the
Fatou property. Indeed, the complex interpolation space
$[L^{p_0(\cdot)}(\R^n),L^{p_1(\cdot)}(\R^n)]_{1-\theta}$ is
isomorphic to the variable Lebesgue space
$L^{p_\theta(\cdot)}(\R^n)$ (see \cite[Theorem~7.1.2]{DHHR11}),
and $L^{p_j(\cdot)}(\R^n)$ have the Fatou property (see
\cite[p.~77]{DHHR11}).

The following characterization of the class $\cB_M(\R^n)$ was communicated to the authors
of \cite{KS13} by Diening.
\begin{theorem}[{\cite[Theorem~4.1]{KS13}}]
\label{th:Deining-interpolation}
If $p\in\cB_M(\R^n)$, then there exist constants $p_0\in(1,\infty)$, $\theta\in(0,1)$,
and a variable exponent $p_1\in\cB_M(\R^n)$ such that
\begin{equation}\label{eq:Diening-interpolation}
\frac{1}{p(x)}=\frac{\theta}{p_0}+\frac{1-\theta}{p_1(x)},
\quad x\in\R^n.
\end{equation}
\end{theorem}
From the above two theorems we obtain the following result, which allows us to transfer
the compactness property from standard Lebesgue spaces to variable Lebesgue spaces.
\begin{lemma}\label{le:compactnesss-transfer}
Let $A\in\cL(L^{p(\cdot)}(\R^n))$ for all $p\in\cB_M(\R^n)$. If $A\in\cK(L^r(\R^n))$ for
some $r\in(1,\infty)$, then $A\in\cK(L^{p(\cdot)}(\R^n))$ for all $p\in\cB_M(\R^n)$.
\end{lemma}
\begin{proof}
By the hypothesis, the operator $A$ is bounded on all standard Lebesgue spaces $L^r(\R^n)$
with $1<r<\infty$. From the classical Krasnosel'skii interpolation theorem
(Theorem~\ref{th:interpolation}(b) with constant exponents) it follows that $A\in\cK(L^r(\R^n))$
for all $1<r<\infty$. If $p\in\cB_M(\R^n)$, then in view of Theorem~\ref{th:Diening}
there exist $p_0\in(1,\infty)$, $\theta\in(0,1)$, and a variable exponent $p_1\in\cB_M(\R^n)$
such that \eqref{eq:Diening-interpolation} holds. Since $A\in\cL(L^{p_1(\cdot)}(\R^n))$ and
$A\in\cK(L^{p_0}(\R^n))$, from Theorem~\ref{th:interpolation}(b) we obtain $A\in\cK(L^{p(\cdot)}(\R^n))$.
\end{proof}
\subsection{Compactness of pseudodifferential operators with $L^\infty(\R,V(\R))$ symbols on variable Lebesgue spaces}
Combining Corollary~\ref{co:boundedness-PDO-Lp} and Theorem~\ref{th:compactness-PDO-Lr} with
Lemma~\ref{le:compactnesss-transfer}, we arrive at our last result.
\begin{corollary}\label{co:compactness-PDO}
Suppose $p\in\cB_M(\R)$. If $a\in L^\infty(\R,V(\R))$ satisfies the hypotheses {\rm (a)--(c)}
of Theorem~\ref{th:compactness-PDO-Lr}, then the pseudodifferential operator $a(x,D)$ is
compact on the variable Lebesgue space $L^{p(\cdot)}(\R)$.
\end{corollary}
\bibliographystyle{amsalpha}

\begin{thebibliography}{XXXXX}
\bibitem[BS88]{BS88}
C. Bennett and R. Sharpley,
\textit{Interpolation of Operators}.
Academic Press, New York, 1988.

\bibitem[CKS92]{CKS92}
F. Cobos, T. K\"uhn, and T. Schonbek,
\textit{One-sided compactness results for Aronszajn-Gagliardo functors.}
\href{http://dx.doi.org/10.1016/0022-1236(92)90049-O}
{J. Funct. Analysis \textbf{106} (1992), 274--313}.

\bibitem[CF13]{CF13}
D. Cruz-Uribe and A. Fiorenza,
\textit{Variable Lebesgue Spaces}.
\href{http://dx.doi.org/10.1007/978-3-0348-0548-3}
{Birkh\"auser, Basel, 2013.}

\bibitem[DPL13]{DPL13}
F. Di Plinio and A. K. Lerner,
\textit{On weighted norm inequalities for the Carleson and Walsh-Carleson operators.}
\href{http://arxiv.org/abs/1312.0833}{arXiv:1312.0833 (2013).}

\bibitem[D05]{D05}
L. Diening,
\textit{Maximal function on Musielak-Orlicz spaces and generlaized Lebesgue spaces.}
\href{http://dx.doi.org/10.1016/j.bulsci.2003.10.003}
{Bull. Sci. Math.  \textbf{129}  (2005),  657--700.}

\bibitem[DHHR11]{DHHR11}
L. Diening, P. Harjulehto, P. H\"ast\"o, and M. R\r{u}\v zi\v cka,
\textit{Lebesgue and Sobolev Spaces with Variable Exponents}.
\href{http://dx.doi.org/10.1007/978-3-642-18363-8}
{Lecture Notes in Mathematics \textbf{2017}, Springer, Berlin, 2011.}

\bibitem[D91]{D91}
E. M. Dyn'kin,
\textit{Methods of the theory of singular integrals (the Hilbert transform and Calderón-Zygmund theory).}
In ``Commutative Harmonic Analysis I",
\href{http://mi.mathnet.ru/rus/intf/v15/p197}
{Itogi Nauki i Tekhniki. Ser. Sovrem. Probl. Mat. Fund. Napr., \textbf{15}, VINITI, Moscow (1987), 197–-292} (in Russian).
English translation: Commutative harmonic analysis I. General survey. Classical aspects, Encycl. Math. Sci. \textbf{15} (1991), 167--259.

\bibitem[FS72]{FS72}
Ch.~Fefferman and E.~M.~Stein,
\textit{$H^p$ spaces of several variables.}
\href{http://dx.doi.org/10.1007/BF02392215}
{Acta Math. \textbf{129} (1972), 137--193.}

\bibitem[G09]{G09}
L. Grafakos,
\textit{Modern Fourier Analysis.} 2nd ed.
\href{http://dx.doi.org/10.1007/978-0-387-09434-2}
{Graduate Texts in Mathematics \textbf{250}. New York, Springer, 2009.}

\bibitem[GMS05]{GMS05}
L. Grafakos, J. M. Martell, and F. Soria,
\textit{Weighted norm inequalities for maximally modulated singular integral operators.}
\href{http://dx.doi.org/10.1007/s00208-004-0586-2}
{Math. Ann. \textbf{331} (2005), 359--394.}

\bibitem[K-A14]{K-A14}
A. Yu. Karlovich,
\textit{Boundedness of pseudodifferential operators on Banach function spaces.}
In: ``Operator Theory, Operator Algebras and Applications".
\href{http://dx.doi.org/10.1007/978-3-0348-0816-3_10}
{Operator Theory: Advances and Applications \textbf{242} (2014), 185--195.}

\bibitem[KS13]{KS13}
A. Yu. Karlovich and I. M. Spitkovsky,
\textit{Pseudodifferential operators on variable Lebesgue spaces.}
In: ``Operator Theory, Pseudo-Differential Equations, and Mathematical Physics. The Vladimir Rabinovich Anniversary Volume".
\href{http://dx.doi.org/10.1007/978-3-0348-0537-7_9}
{Operator Theory: Advances and Applications \textbf{228} (2013), 173--183.}

\bibitem[KS14]{KS14}
A. Yu. Karlovich and I. M. Spitkovsky,
\textit{The Cauchy singular integral operator on weighted variable Lebesgue spaces.}
In: ``Concrete Operators, Spectral Theory, Operators in Harmonic Analysis and Approximation".
\href{http://dx.doi.org/10.1007/978-3-0348-0648-0_17}
{Operator Theory: Advances and Applications \textbf{236} (2014), 275--291.}

\bibitem[K-Yu07]{K-Yu07}
Yu. I. Karlovich,
\textit{Algebras of pseudo-differential operators with discontinuous symbols.}
In: ``Modern Trends in Pseudo-Differential Operators".
\href{http://dx.doi.org/10.1007/978-3-7643-8116-5_12}
{Operator Theory: Advances and Applications \textbf{172} (2007), 207--233.}

\bibitem[K-Yu12]{K-Yu12}
Yu. I. Karlovich,
\textit{Boundedness and compactness of pseudodifferential operators with non-regular symbols on weighted Lebesgue spaces.}
\href{http://dx.doi.org/10.1007/s00020-012-1951-2}
{Integr. Equ. Oper. Theor. \textbf{73} (2012), 217--254.}

\bibitem[L05]{L05}
A. K. Lerner,
\textit{Some remarks on the Hardy-Littlewood maximal function on variable $L^p$ spaces.}
\href{http://dx.doi.org/10.1007/s00209-005-0818-5}
{Math. Z. \textbf{251} (2005), 509–-521.}

\bibitem[L10]{L10}
A. K. Lerner,
\textit{Some remarks on the Fefferman-Stein inequality.}
\href{http://dx.doi.org/10.1007/s11854-010-0032-1}
{J. Anal. Math. \textbf{112} (2010), 329-–349.}

\bibitem[LO10]{LO10}
A. K. Lerner and S. Ombrosi,
\textit{A boundedness criterion for general maximal operators.}
\href{http://dx.doi.org10.5565/PUBLMAT_54110_03}
{Publ. Mat. \textbf{54} (2010), 53–-71.}

\bibitem[LP07]{LP07}
A. K. Lerner and C. P\'erez,
\textit{A new characterization of the Muckenhoupt $A_p$ weights through an extension of the Lorentz-Shimogaki theorem.}
\href{http://dx.doi.org/10.1512/iumj.2007.56.3112}
{Indiana Univ. Math. J. \textbf{56} (2007), 2697-–2722.}

\bibitem[MS13]{MS13}
C. Muscalu and W. Schlag,
\textit{Classical and Multilinear Harmonic Analysis.} Vol. II.
\href{ http://dx.doi.org/10.1017/CBO9781139410397}
{Cambridge Studies in Advanced Mathematics \textbf{138}. Cambridge University Press, Cambridge, 2013.}

\bibitem[N55]{N55}
I. P. Natanson,
\textit{Theory of Functions of a Real Variable.}
Frederick Ungar Publishing Co., New York, 1955.

\bibitem[S93]{S93}
E. Stein,
\textit{Harmonic Analysis: Real-Variable Methods, Orthogonality, and Oscillatory Integrals.}
Princeton Uinversity Press, Princeton, NJ, 1993.

\bibitem[S67]{S67}
L. Schwartz,
\textit{Analyse Math\'ematique. Cours I.}
Hermann, Paris, 1967.
\end{thebibliography}

\end{document}